\documentclass[12pt]{article}

\usepackage{url}
\usepackage{amsthm}
\usepackage{amsmath}
\usepackage{fullpage}
\usepackage{enumerate}
\usepackage{amssymb}

\usepackage{booktabs}
\usepackage{float}

\newtheorem{thm}{Theorem}
\newtheorem{lem}{Lemma}

\newcommand{\e}{\rm e}

\title{Fujii's development on Chebyshev's conjecture}
\author{Dave Platt\footnote{Supported by Australian Research Council Discovery Project DP160100932 and  EPSRC Grant EP/K034383/1.}\\ School of Mathematics \\ University of Bristol, Bristol, UK\\dave.platt@bris.ac.uk\and  Tim Trudgian\footnote{Supported by Australian Research Council Discovery Project DP160100932 and Future Fellowship FT160100094.} \\
School of Physical, Environmental and Mathematical Sciences\\ The University of New South Wales Canberra, Australia \\
  t.trudgian@adfa.edu.au}

\begin{document}
\maketitle
\begin{abstract}
\noindent
Chebyshev presented a conjecture after observing the apparent bias towards primes congruent to $3\pmod 4$. His conjecture is equivalent to a version of the Generalised Riemann Hypothesis. Fujii strengthened this conjecture; we strengthen it still further using detailed computations of zeroes of Dirichlet $L$-functions.
  \end{abstract}

\section{Introduction}
Chebyshev observed that there appear to be more primes congruent to $3\pmod 4$ that $1\pmod 4$. This bias has spawned much research --- see, e.g. the seminal work by Rubinstein and Sarnak \cite{Rubin}, and also Ford and Konyagin \cite{Ford}. In 1853 Chebyshev conjectured that
\begin{equation}\label{cook}
\sum_{p>2} (-1)^{(p-1)/2} e^{-xp} \rightarrow -\infty,
\end{equation}
as $x\rightarrow 0$. Hardy and Littlewood \cite{Hardy} and Landau \cite{Landau} showed that (\ref{cook}) is equivalent to all of the non-trivial zeroes of $L(s, \chi_{4})$ having real part $\sigma = \frac{1}{2}$, where we use $\chi_{4}$ to denote the non-principal Dirichlet character modulo 4. We shall refer to this specialised version of the Generalised Riemann Hypothesis for $\chi_{4}$ as `GRH for $\chi_{4}$'. 

%Knapowski and Tur\'{a}n \cite{Nap} considered a variant of (\ref{cook}), which was later investigated by Bentz \cite{Bentz}. 
In this article we examine the attenuation factor $e^{-xp}$ in (\ref{cook}). Fujii \cite[Thm 1]{Fujii} showed that for all $0<\alpha< 4.19$  the statement
\begin{equation}\label{stoneman}
\sum_{p>2} (-1)^{(p-1)/2} e^{-(xp)^{\alpha}} \rightarrow -\infty \; \textrm{as}\; x\rightarrow 0 
\end{equation}
is equivalent to GRH for $\chi_{4}$. Note that the larger one can take $\alpha$ the quicker the summands in (\ref{stoneman}) attenuate, and hence there must be an even greater bias towards primes congruent to $3\pmod 4$.

Fujii's argument is elegant; his result of $\alpha<4.19$ is a result of some numerical calculations involving the first few zeroes $\beta + i \gamma$ of $L(s, \chi_{4})$. In fact, Fujii only uses the fact that $\gamma_{1}>6$ and that  $\sum_{\gamma>0} \gamma^{-2} < 1/5$.

We use some more extensive calculations on the zeroes of $L(s, \chi_{4})$ and some optimisation to improve Fujii's work. The result is the following theorem.
\begin{thm}\label{root}
Suppose that $0<\alpha< 20.40442$. Then the statement
\begin{equation*}
\sum_{p>2} (-1)^{(p-1)/2} e^{-(xp)^{\alpha}} \rightarrow -\infty \; \textrm{as}\; x\rightarrow 0 
\end{equation*}
is equivalent to all of the non-trivial zeroes of $L(s, \chi)$ having real part $\sigma = \frac{1}{2}$, where $\chi$ is the non-principal Dirichlet character modulo 4.
\end{thm}
We introduce Fujii's work in \S \ref{rum}, and prove Theorem \ref{root} in \S \ref{coke}. We remark at the end of \S \ref{coke} that it appears impossible to improve Theorem \ref{rum} further using Fujii's method.
\section{Fujii's method and some lemmas}\label{rum}
Proceeding as in Fujii \cite[\S 3]{Fujii}, we have, under the assumption of GRH for $\chi_{4}$,
\begin{equation*}
S = \sum_{p>2} (-1)^{(p-1)/2} e^{-(xp)^{\alpha}} = S_{1} + S_{2} + S_{3} + S_{4},
\end{equation*}
where
\begin{equation*}
\begin{split}
S_{1} &= -\frac{1}{2} \Gamma\left(\frac{1}{2\alpha}\right) x^{-1/2} + o(x^{-1/2})\\
S_{2}+ S_{4} &= o(x^{-1/2})\\
|S_{3}| &\leq x^{-1/2} \sum_{\rho} \left| \Gamma\left(\frac{1}{2\alpha} + \frac{i\gamma}{\alpha}\right)\right|.
\end{split}
\end{equation*}
Therefore, to show that $S\rightarrow -\infty$ as $x\rightarrow 0$ it is sufficient to show that
\begin{equation}\label{stokes?}
\sum_{\rho} \left| \Gamma\left(\frac{1}{2\alpha} + \frac{i\gamma}{\alpha}\right)\right| < \frac{1}{2}\Gamma\left(\frac{1}{2\alpha}\right). 
\end{equation}
We require an explicit version of Stirling's formula to bound the summands in (\ref{stokes?}). Many versions abound in the literature: we shall use the one given by Olver \cite[p.\ 294]{Olver}, namely
\begin{equation}\label{vince_ha!}
\log \Gamma(z) = \left(z-\frac{1}{2}\right) \log z - z + \frac{1}{2} \log 2\pi + \frac{\vartheta}{6|z|}, \quad (|\arg z|\leq \frac{\pi}{2}).
\end{equation}
Using (\ref{vince_ha!}) and (\ref{stokes?}) and that fact that $\tan^{-1}x < x$ for all $x$, we see that we have
\begin{equation*}\label{anderson}
 \left| \Gamma\left(\frac{1}{2\alpha} + \frac{i\gamma}{\alpha}\right)\right| \leq (2\pi)^{1/2} \left( \frac{\sqrt{\gamma^{2} + \frac{1}{4}}}{\alpha}\right)^{\frac{1}{2\alpha} - \frac{1}{2}} \exp\left( -\frac{\pi \gamma}{2\alpha} + \frac{\alpha}{6 \sqrt{\gamma^{2} + \frac{1}{4}}}\right).
\end{equation*}
%For $\gamma$ large the summands in (\ref{anderson}) will look like
%\begin{equation}\label{broad}
%\left(\frac{\alpha}{\gamma}\right)^{\frac{1}{2} - \frac{1}{2\alpha}} \exp\left( -\frac{\pi \gamma}{2\alpha}\right).
%\end{equation}
We aim at writing the sum in  (\ref{stokes?}) as
$\Sigma = \Sigma_{1} + \Sigma_{2}$ where $0<\gamma \leq T_{1}$ in $\Sigma_{1}$ and $\gamma> T_{1}$ in $\Sigma_{2}$. We shall choose $T_{1}$ such that we have detailed information on the location of zeroes with $\gamma \leq T_{1}$. We shall sum the contribution from these zeroes explicitly. We shall then estimate $\Sigma_{2}$ using (\ref{anderson}) and  bounds on $N(T, \chi_{4})$, the number of zeroes of $L(s, \chi_{4})$ with $|\gamma|\leq T$. We have such an estimate in \cite{Trud}, 
%(valid for any non-principal character modulo $q$)
namely, that
\begin{equation}\label{gooch}
|N(T, \chi_{4}) - \frac{T}{\pi} \log \frac{2 T}{\pi e}| \leq C_{1} \log 4T + C_{2}, \quad (T\geq 1),
\end{equation}
where $C_{1}$ and $C_{2}$ are explicitly given constants. Note that the definition of $N(T, \chi)$ counts zeroes with $|\gamma| \leq T$. We actually wish to count the zeroes with $\gamma\geq 0$. Therefore, we divide (\ref{gooch}) by 2, giving us
\begin{equation*}
N(T, \chi) = \frac{T}{2\pi} \log \frac{2 T}{\pi} - \frac{T}{2\pi} + Q(T),
\end{equation*}
where
\begin{equation*}
|Q(T)| \leq \frac{C_{1}}{2} \log 4 T + \frac{C_{2}}{2} \leq \theta_{1}\log T, \quad (T\geq T_{1})
\end{equation*}
say. Henceforth we consider everything in terms of $T_{1}$, which will be the truncation point in the sum. 
We need the following, which is a trivial adaptation of a result by Lehman.
\begin{lem}\label{underwood}
Let $\phi(t)$ be a decreasing function with continuous derivative on $[T_{1}, T_{2}]$. For $L(s, \chi)$ the non-principal $L$-function with $\chi$ to the modulus $4$ we have, for any $T_{1}\geq 1$ that
\begin{equation*}
\sum_{T_{1} < \gamma \leq T_{2}} \phi(\gamma) = \frac{1}{2\pi} \int_{T_{1}}^{T_{2}} \phi(t) \log \frac{2t}{\pi} \, dt + \theta_{1}\left\{ 2 \phi(T_{1}) \log T_{1} + \int_{T_{1}}^{T_{2}} \frac{\phi(t)}{t} \, dt \right\},
\end{equation*}
where $C_{1}$ and $C_{2}$ are in (\ref{gooch}) and $\theta_{1}$ is such that
\begin{equation}\label{hemmings}
\theta_{1} \geq \frac{\frac{C_{1}}{2} \log 4 T_{1} + \frac{C_{2}}{2}}{\log T_{1}}.
\end{equation}
\end{lem}
\begin{proof}
The proof follows the proof given in Lehman \cite[p.\ 400]{Lehman}.
\end{proof}
We now apply Lemma \ref{underwood} with $$\phi(t) = \exp\left\{ \frac{-\pi t}{2\alpha}\right\}$$ and send $T_{2}\rightarrow\infty$. We obtain
\begin{equation}\label{russell}
\sum_{\gamma> T_{1}} \phi(\gamma)\leq\exp\left(-\frac{\pi T_1}{2\alpha}\right)\left[\frac{\alpha}{\pi^2}\log\frac{2T_1\e}{\pi}+\theta_1\log T_1^2\e\right].
\end{equation}
Putting this together with (\ref{anderson}) we find that $\alpha$ is admissible in Theorem \ref{root} if
\begin{equation}\label{sutcliffe}
\begin{split}
&   2(2\pi)^{1/2} \left(\frac{\sqrt{T_1^2+\frac{1}{4}}}{\alpha}\right)^{\frac{1}{2\alpha}-\frac12}\exp\left(\frac{\alpha}{6\sqrt{T_1^2+\frac{1}{4}}}\right)\exp\left(-\frac{\pi T_1}{2\alpha}\right)\left[\frac{\alpha}{\pi^2}\log\frac{2T_1\e}{\pi}+\theta_1\log T_1^2\e\right]\\
&+2\sum_{0<\gamma< T_{1}} \left|\Gamma\left(\frac{1}{2\alpha}+\frac{i \gamma}{\alpha}\right)\right| \leq \frac{1}{2}\Gamma\left(\frac{1}{2\alpha}\right),
\end{split}
\end{equation}
subject to (\ref{hemmings}). We replicate (\ref{hemmings}) here for convenience, and choose $C_{1} = 0.315$ and $C_{2} = 6.445$ as in \cite{Trud}. Therefore, we require that (\ref{sutcliffe}) be satisfied along with
\begin{equation*}
\theta_{1} = \frac{0.1575 \log 4 T_{1} + 3.2225}{\log T_{1}}.
\end{equation*}

\section{Computations and proof of Theorem \ref{root}}\label{coke}

We used ``lcalc'' \cite{lcalc} to produce a list of the lowest $1\,000$ zeroes of $L(s,\chi_4)$. The output gives $11$ decimal places reducing to $10$ for the highest zeroes. We checked each $t$ actually did represent a zero by using ARB \cite{ARB} to rigorously compute
$$
\frac{4}{\pi}^{s/2}\Gamma\left(\frac{s+1}{2}\right)4^{-s}\left[\zeta\left(s,\frac{1}{4}\right)-\zeta\left(s,\frac{3}{4}\right)\right]
$$
with $s=1/2+i(t-\delta)$ and $s=1/2+i(t+\delta)$ with $\delta=10^{-10}$ and checking that we saw a sign change in every case. We then use a rigorous version of Turing's method \cite{Booker} to confirm that ``lcalc'' had (as expected) found all the zeroes with $\Im \rho \in [0,1\,127]$.

Taking these lowest $1\,000$ zeroes we can set $T_1=1\,127$ and we see that $\alpha=20.40442$ gives us
$$
\sum\limits_{|\Im \rho |<T_1} \left|\Gamma\left(\frac{1}{2\alpha}+\frac{i \gamma}{\alpha}\right)\right|< 20.1276643
$$
whereas
$$
\frac{1}{2}\Gamma\left(\frac{1}{2\alpha}\right)> 20.1276649.
$$
We also find that the contribution from zeroes with $|\Im\rho| > T_1$ is strictly less in absolute terms than $4\times 10^{-38}$ so $\alpha=20.40442$ is admissible.

Further, if we take $\alpha=20.40443$ we find that the sum over the zeroes with imaginary part $<1\,127$ exceeds $\Gamma(1/(2\alpha))/2$ even if we ignore the contribution from the rest of the zeroes. Thus regardless of how many more zeroes we consider, how precisely we know their imaginary parts or how small we can make the constant $\theta_1$, we will never be able to show that $\alpha=20.40443$ is admissible by this method.


\begin{thebibliography}{99}
%\bibitem{Bentz}
%Bentz, H.-J.
%\newblock Discrepancies in the distirbution of primes numbers.
%\newblock {\em J. Number Theory}, 15(2), 252--274, 1982.

\bibitem{Booker}
Booker, A.R.\
\emph{{Artin's conjecture, Turing's method and the Riemann
  hypothesis}}, Exp. Math. 15(4), 385--407, 2006.

\bibitem{Ford}
Ford, K. and Konyagin, S.
\emph{{Chebyshev's conjecture and the prime number race}},
IV International Conference ``Modern Problems of Number Theory and its Applications'': Current Problems, Part II (Russian) (Tula, 2001), 67Ð91, Mosk. Gos. Univ. im. Lomonosova, Mekh.-Mat. Fak., Moscow, 2002. 

\bibitem{Fujii}
Fujii, A.
\emph{{Some generalizations of Chebyshev's conjecture}},
Proc. Japan Acad. Ser. A Math. Sci. 64(7), 260--263, 1988.

\bibitem{Hardy}
Hardy, G.H.\ and Littlewood, J.E.\
\emph{{Contribution to the theory of the Riemann zeta function and the theory of the distribution of primes}},
Acta Math. 41, 119--196, 1917.

\bibitem{ARB}
Johansson, F.
\emph{{Arb: efficient arbitrary-precision midpoint-radius interval arithmetic}},
IEEE Transactions on Computers, 66(8), 1281--1292, 2017.

\bibitem{Landau}
Landau, E. 
\emph{{\"{U}ber einige altere Vermutungen und Behauptungen in der Primzhaltheorie I \& II}},
Math. Z. 1, 1--24, 213--219, 1918.

\bibitem{Lehman}
Lehman, R. S.
\emph{{On the difference $\pi(x) - \textrm{li}(x)$}},
Acta Arith. 11, 397--410, 1966.

\bibitem{Olver}
Olver, F.~W.~J.\ 
\emph{{Asymptotics and Special Functions}},
Computer Science and Applied Mathematics. Academic Press, New York--London, 1974.

\bibitem{lcalc}
Rubinstein, M.
\emph{{lcalc -- The L-function Calculator}},
\url{ http://code.google.com/p/l-calc/}.

\bibitem{Rubin}
Rubinstein, M. and Sarnak, P.
\emph{{Chebyshev's bias}},
Exp. Math. 3(3), 173--197, 1994

\bibitem{Trud}
Trudgian, T. S.
\emph{{An improved upper bound for the error in the zero-counting formulae for Dirichlet $L$-functions and Dedekind zeta-functions}},
Math. Comp. 84(293), 1439--1450, 2015.



%\bibitem{Nap}
%Knapowski, S. and Tur\'{a}n, P.







\end{thebibliography}
\end{document}